\documentclass[11pt]{article}
\usepackage{amsmath}
\usepackage{amssymb}
\usepackage{amsthm}
\usepackage{bookmark}
\usepackage{geometry}
\usepackage{hyperref}
\newtheorem{theorem}{Theorem}
\newtheorem{corollary}[theorem]{Corollary}
\newtheorem{lemma}[theorem]{Lemma}
\newtheorem{proposition}[theorem]{Proposition}
\theoremstyle{definition}
\newtheorem{remark}[theorem]{Remark}
\DeclareMathOperator{\im}{im}

\title{Permutation Polynomials of Finite Fields of Even Characteristic From Character Sums}
\author{Ruikai Chen\textsuperscript{1}\and Sihem Mesnager\textsuperscript{2,3,4}}
\date{\small\textsuperscript{1}School of Mathematical Sciences, South China Normal University, 510631 Guangzhou, China\\\textsuperscript{2}Department of Mathematics, University of Paris VIII, 93526 Saint-Denis, France\\\textsuperscript{3}Laboratory Analysis, Geometry and Applications, LAGA, University Sorbonne Paris Nord, CNRS, UMR 7539, 93430 Villetaneuse, France\\\textsuperscript{4}Telecom Paris, Polytechnic institute of Paris, 91120 Palaiseau, France\\Emails: \href{mailto:chen.rk@outlook.com}{chen.rk@outlook.com}\quad\href{mailto:smesnager@univ-paris8.fr}{smesnager@univ-paris8.fr}}

\begin{document}

\maketitle

\begin{abstract}
In this paper, we investigate permutation polynomials over the finite field $\mathbb F_{q^n}$ with $q=2^m$, focusing on those in the form $\mathrm{Tr}(Ax^{q+1})+L(x)$, where $A\in\mathbb F_{q^n}^*$ and $L$ is a $2$-linear polynomial over $\mathbb F_{q^n}$. By calculating certain character sums, we characterize these permutation polynomials and provide additional constructions.

\end{abstract}

\section{Introduction}

A polynomial over the finite field $\mathbb F_q$ whose associated map is a permutation of $\mathbb F_q$ is called a permutation polynomial of $\mathbb F_q$. Given a map from $\mathbb F_q$ to itself, the corresponding polynomial can be retrieved using the Lagrange interpolation formula. In general, there exists a one-to-one correspondence between all polynomials over $\mathbb F_q$ of degree less than $q$ and those maps on $\mathbb F_q$.

Permutation polynomials, especially those over finite fields of characteristic $2$, have significant applications in cryptography, coding theory, combinatorics, etc. The problem of finding permutation polynomials in certain forms has gained substantial attention in recent decades. Specifically, permutation polynomials that are so-called Dembowski-Ostrom polynomials over finite fields of characteristic $2$ have been studied in \cite{blokhuis2001permutations}, such as $x^{2^{2k}+1}+a^{2^k+1}x^{2^k+1}+ax^2$ over $\mathbb F_{2^{3k}}$ for a positive integer $k$. Some other related permutation polynomials, to name a few, are listed as follows:
\begin{itemize}
\item $x^{2^{2k}+1}+x^{2^k+1}+ax$ over $\mathbb F_{2^{3k}}$ in \cite{tu2014several};
\item $x^{2^s+1}+x^{2^{s-1}+1}+x$ over $\mathbb F_{2^k}$ with $s\in\{1,2\}$ in \cite{bhattacharya2017some};
\item $x^rh(x^{2^k-1})$ over $\mathbb F_{2^{2k}}$ in \cite{gupta2016some,zha2017further}.
\end{itemize}
See \cite{ding2015permutation,li2017new,wang2018six} and the references therein for more examples.

Most permutation polynomials studied in the literature have specific forms with several coefficients. In this paper, we aim to provide a systematic treatment for a wider range of polynomials. By introducing certain character sums, we derive criteria to determine whether a given polynomial is a permutation polynomial. This approach may lead to new perspectives for studying permutation polynomials and their potential applications.

Throughout this paper, we consider the finite field $\mathbb F_{q^n}$ as an extension of $\mathbb F_q$ with $q=2^m$. For the polynomial ring $\mathbb F_{q^n}[x]$, when polynomials are viewed as maps on $\mathbb F_{q^n}$, let $x^{q^n}=x$ by abuse of notation. In this sense, those polynomials over $\mathbb F_{q^n}$ representing linear endomorphisms of $\mathbb F_{q^n}/\mathbb F_q$ are called $q$-linear polynomials, where each term has degree a power of $q$. Also, denote by $x^\frac12$ the inverse of the automorphism $x^2$ of $\mathbb F_{q^n}$, which is $x^\frac{q^n}2$ written as a polynomial over $\mathbb F_{q^n}$. Denote by $\mathrm{Tr}_k$ the trace map of $\mathbb F_{q^n}$ over $\mathbb F_{q^k}$ if $k$ divides $n$, and let $\mathrm{Tr}=\mathrm{Tr}_1$.

We will deal with the polynomial
\[\mathrm{Tr}(Ax^{q+1})+L(x)\]
for $A\in\mathbb F_{q^n}^*$ and $L$ a $2$-linear polynomial over $\mathbb F_{q^n}$. In Section \ref{pre}, some preliminaries on linear structures of finite fields and character sums will be introduced. Afterwards, we study those polynomials, providing necessary and sufficient conditions for them to be permutation polynomials of $\mathbb F_{q^n}$, along with various constructions. These will be presented in Section \ref{odd} for the case of odd $n$ and in Section \ref{even} for even $n$.

\section{Preliminaries}\label{pre}

Consider $\mathbb F_{q^n}$ as a vector space over $\mathbb F_q$. Let $W$ be a subspace of $\mathbb F_{q^n}$ with a basis $\beta_1,\dots,\beta_r$. The map defined by
\[\alpha\mapsto(\mathrm{Tr}(\beta_1\alpha),\dots,\mathrm{Tr}(\beta_r\alpha))\qquad(\alpha\in\mathbb F_{q^n})\]
is a homomorphism from $\mathbb F_{q^n}$ onto $\mathbb F_q^r$, and its kernel, denoted by $W^\perp$, does not depend on the choice of basis. In particular, when $r=n$, that is an isomorphism from $\mathbb F_{q^n}$ to $\mathbb F_q^n$.

Given a linear endomorphism $L$ of $\mathbb F_{q^n}/\mathbb F_q$ of the form
\[L(x)=\sum_{i=0}^{n-1}a_ix^{q^i},\]
define $L^\prime$ as
\[L^\prime(x)=\sum_{i=0}^{n-1}(a_ix)^{q^{-i}},\]
called the adjoint of $L$. It is the unique map on $\mathbb F_{q^n}$ such that $\mathrm{Tr}(\alpha L(\beta))=\mathrm{Tr}(L^\prime(\alpha)\beta)$ for all $\alpha,\beta\in\mathbb F_{q^n}$, where $\mathrm{Tr}$ acts as a linear form of $\mathbb F_{q^n}/\mathbb F_q$. In fact, for every fixed $\alpha\in\mathbb F_{q^n}$, $L^\prime(\alpha)$ is uniquely determined by $\mathrm{Tr}(L^\prime(\alpha)\beta)$ with $\beta$ ranging over a basis of $\mathbb F_{q^n}/\mathbb F_q$ via the aforementioned isomorphism. In addition, the adjoints of linear endomorphisms have the following properties:
\begin{enumerate}
\item[(i)]$(L^\prime)^\prime=L$;
\item[(ii)]$(L+L_0)^\prime=L^\prime+L_0^\prime$, for another linear endomorphism $L_0$ of $\mathbb F_{q^n}/\mathbb F_q$;
\item[(iii)]$(L_0\circ L)^\prime=L^\prime\circ L_0^\prime$;
\item[(iv)]$(L^{-1})^\prime=(L^\prime)^{-1}$ if $L$ has an inverse map $L^{-1}$ on $\mathbb F_{q^n}$;
\item[(v)]$|\ker L|=|\ker L^\prime|$.
\end{enumerate}
Here (i) and (ii) are obvious from the definition and (iii) follows from the identity
\[\mathrm{Tr}(\alpha L^\prime(L_0^\prime(\beta)))=\mathrm{Tr}(L(\alpha)L_0^\prime(\beta))=\mathrm{Tr}(L_0(L(\alpha))\beta)\]
for all $\alpha,\beta\in\mathbb F_{q^n}$. This immediately implies $(L^{-1})^\prime\circ L^\prime=(L\circ L^{-1})^\prime$ in the case of (iv). The last one is guaranteed by the fact that $\ker L^\prime=(\im L)^\perp$ ($\im L$ denotes the image of $L$).

Since $q$ and $n$ are arbitrary, the above properties are still valid concerning general $2$-linear polynomials over $\mathbb F_{q^n}$ and the trace map of $\mathbb F_{q^n}$ over $\mathbb F_2$.

Given $A\in\mathbb F_{q^n}^*$ and an arbitrary $2$-linear polynomial $L$ over $\mathbb F_{q^n}$, we are interested in the number of roots of $\mathrm{Tr}(Ax^{q+1})+L(x)$ in $\mathbb F_{q^n}$. The above discussion indicates that $\chi(\alpha L(\beta))=\chi(L^\prime(\alpha)\beta)$ for all $\alpha,\beta\in\mathbb F_{q^n}$, where $\chi$ is the canonical additive character of $\mathbb F_{q^n}$. Let
\[\mathcal S(a,b)=\sum_{w\in\mathbb F_{q^n}}\chi(aw^{q+1}+bw)\]
for $a,b\in\mathbb F_{q^n}$. Then the number of roots of $\mathrm{Tr}(Ax^{q+1})+L(x)$ in $\mathbb F_{q^n}$ is
\[q^{-n}\sum_{w\in\mathbb F_{q^n}}\sum_{u\in\mathbb F_{q^n}}\chi(u(\mathrm{Tr}(Aw^{q+1})+L(w))),\]
for the inner sum is $q^n$ if $\mathrm{Tr}(Aw^{q+1}+L(w))=0$ and is $0$ otherwise. Hence, the desired number is
\[q^{-n}\sum_{u\in\mathbb F_{q^n}}\sum_{w\in\mathbb F_{q^n}}\chi(A\mathrm{Tr}(u)w^{q+1}+L^\prime(u)w)=q^{-n}\sum_{u\in\mathbb F_{q^n}}\mathcal S(A\mathrm{Tr}(u),L^\prime(u)).\]
These character sums for odd $n$ and for even $n$ will be treated differently.

\section{The case of odd $n$}\label{odd}

For odd $n$, one has
\[q^n-1=(q+1-1)^n-1\equiv(-1)^n-1\equiv-2\pmod{q+1},\]
so $\gcd(q+1,q^n-1)=\gcd(2,q+1)=1$. This means for every $A\in\mathbb F_{q^n}^*$ we may write
\[\mathrm{Tr}(Ax^{q+1})+L(x)=\mathrm{Tr}((\alpha x)^{q+1})+L(x)\]
for some $\alpha\in\mathbb F_{q^n}^*$. Then consider $\mathrm{Tr}(x^{q+1})+L(x)$ without loss of generality. It is a permutation polynomial of $\mathbb F_{q^n}$ if and only if for every $\beta\in\mathbb F_{q^n}$ the map
\[\mathrm{Tr}((x+\beta)^{q+1})+L(x+\beta)+\mathrm{Tr}(\beta^{q+1})+L(\beta)\]
has only one root in $\mathbb F_{q^n}$, and that is
\[\mathrm{Tr}(x^{q+1})+\mathrm{Tr}(\beta x^q+\beta^qx)+L(x)=\mathrm{Tr}(x^{q+1})+\mathrm{Tr}\big(\big(\beta^{q^{n-1}}+\beta^q\big)x\big)+L(x).\]
Note that $x^{q^{n-1}}+x^q$ maps $\mathbb F_{q^n}$ onto $\ker\mathrm{Tr}$ with kernel $\mathbb F_q$, so $\beta^{q^{n-1}}+\beta^q$ runs through $\ker\mathrm{Tr}$ as $\beta$ runs through $\mathbb F_{q^n}$. Then $\mathrm{Tr}(x^{q+1})+L(x)$ is a permutation polynomial of $\mathbb F_{q^n}$ if and only if
\[q^{-n}\sum_{u\in\mathbb F_{q^n}}\mathcal S(\mathrm{Tr}(u),t\mathrm{Tr}(u)+L^\prime(u))=1\]
for every $t\in\ker\mathrm{Tr}$, where the left side is indeed the number of roots of $\mathrm{Tr}(x^{q+1})+\mathrm{Tr}(tx)+L(x)$ in $\mathbb F_{q^n}$. As for the character sum $\mathcal S(a,b)$ with $a,b\in\mathbb F_{q^n}^*$, there exists a unique element $c$ in $\mathbb F_{q^n}$ satisfying $c^{q+1}=a$, so that $\mathcal S(a,b)=\mathcal S(1,bc^{-1})$, which can be determined by the following lemma.

\begin{lemma}[{\cite[Theorem 4.2, Lemma 4.3]{coulter1999evaluation}}]
Suppose that $n$ is odd. For $b\in\mathbb F_{q^n}^*$, if $b+1=\beta^{q^2}+\beta$ for some $\beta\in\mathbb F_{q^n}$, then $\mathcal S(1,b)=\chi(\beta^{q+1}+\beta)\mathcal S(1,1)$ with $\mathcal S(1,1)^2=q^{n+1}$; otherwise, $\mathcal S(1,b)=0$.
\end{lemma}

We are now ready to characterize the polynomial $\mathrm{Tr}(x^{q+1})+L(x)$ and ascertain whether it is a permutation polynomial of $\mathbb F_{q^n}$.

\begin{theorem}\label{t}
Let $n$ be odd and
\[L(x)=\sum_{i=0}^{m-1}L_i\big(x^{2^i}\big),\]
where $L_i$ is a $q$-linear polynomial over $\mathbb F_{q^n}$. Then $\mathrm{Tr}(x^{q+1})+L(x)$ is a permutation polynomial of $\mathbb F_{q^n}$ if and only if $L_i(1)\in\mathbb F_q$ for $0\le i<m$,
\[\ell(x)=\sum_{i=0}^{m-1}(L_i(1)x)^{2^{m-i}}+x^{2^{m-1}}\]
is a permutation polynomial of $\mathbb F_q$, and $\ker\mathrm{Tr}\cap\ker L^\prime=\{0\}$.
\end{theorem}
\begin{proof}
First of all, for $t\in\ker\mathrm{Tr}$ we have
\[\sum_{u\in\ker\mathrm{Tr}}\mathcal S(\mathrm{Tr}(u),t\mathrm{Tr}(u)+L^\prime(u))=\sum_{u\in\ker\mathrm{Tr}}\mathcal S(0,L^\prime(u))=|\ker\mathrm{Tr}\cap\ker L^\prime|q^n.\]
Then consider $u\in\mathbb F_{q^n}$ such that $\mathrm{Tr}(u)\ne0$. Since $\big(\mathrm{Tr}(u)^\frac12\big)^{q+1}=\big(\mathrm{Tr}(u)^\frac12\big)^2=\mathrm{Tr}(u)$, we have
\[\mathcal S(\mathrm{Tr}(u),t\mathrm{Tr}(u)+L^\prime(u))=\mathcal S\big(1,t\mathrm{Tr}(u)^\frac12+L^\prime(u)\mathrm{Tr}(u)^{-\frac12}\big).\]
Also, this character sum is nonzero if and only if
\[\mathrm{Tr}(L^\prime(u))\mathrm{Tr}(u)^{-\frac12}+1=\mathrm{Tr}\big(t\mathrm{Tr}(u)^\frac12+L^\prime(u)\mathrm{Tr}(u)^{-\frac12}+1\big)=0;\]
that is,
\[\mathrm{Tr}(L^\prime(u))+\mathrm{Tr}(u)^\frac12=0\qquad(\mathrm{Tr}(u)\ne0).\]
Let $T$ be a linear endomorphism of $\mathbb F_{q^n}/\mathbb F_2$ defined by $T(x)=\mathrm{Tr}(L^\prime(x))+\mathrm{Tr}(x)^\frac12$, which is
\[T(x)=\sum_{i=0}^{m-1}\mathrm{Tr}(L_i^\prime(x))^{2^{-i}}+\mathrm{Tr}(x)^\frac12=\sum_{i=0}^{m-1}\mathrm{Tr}(L_i(1)x)^{2^{m-i}}+\mathrm{Tr}(x)^{2^{m-1}}.\]
Next, observe that
\[\begin{split}&\mathrel{\phantom{=}}\sum_{t\in\ker\mathrm{Tr}}\sum_{u\in\ker T\setminus\ker\mathrm{Tr}}\mathcal S(\mathrm{Tr}(u),t\mathrm{Tr}(u)+L^\prime(u))\\&=\sum_{u\in\ker T\setminus\ker\mathrm{Tr}}\sum_{t\in\ker\mathrm{Tr}}\mathcal S\big(1,t\mathrm{Tr}(u)^\frac12+L^\prime(u)\mathrm{Tr}(u)^{-\frac12}\big)\\&=\sum_{u\in\ker T\setminus\ker\mathrm{Tr}}q^{-1}\sum_{\beta\in\mathbb F_{q^n}}\mathcal S(1,\beta^{q^2}+\beta+1)\\&=\sum_{u\in\ker T\setminus\ker\mathrm{Tr}}q^{-1}\sum_{\beta\in\mathbb F_{q^n}}\chi(\beta^{q+1}+\beta)\mathcal S(1,1)\\&=|\ker T\setminus\ker\mathrm{Tr}|q^{-1}\mathcal S(1,1)^2\\&=|\ker T\setminus\ker\mathrm{Tr}|q^n,\end{split}\]
where the second equality holds because $t\mathrm{Tr}(u)^\frac12+L^\prime(u)\mathrm{Tr}(u)^{-\frac12}$ runs through $\ker\mathrm{Tr}+1$ as $t$ runs through $\ker\mathrm{Tr}$ for every $u\in\ker T\setminus\ker\mathrm{Tr}$. Moreover, it has been seen that
\[\sum_{t\in\ker\mathrm{Tr}}\sum_{u\in\ker\mathrm{Tr}}\mathcal S(\mathrm{Tr}(u),t\mathrm{Tr}(u)+L^\prime(u))=|\ker\mathrm{Tr}\cap\ker L^\prime|q^{2n-1}.\]
Altogether, we get
\[q^{-n}\sum_{t\in\ker\mathrm{Tr}}\sum_{u\in\mathbb F_{q^n}}\mathcal S(\mathrm{Tr}(u),t\mathrm{Tr}(u)+L^\prime(u))=|\ker\mathrm{Tr}\cap\ker L^\prime|q^{n-1}+|\ker T\setminus\ker\mathrm{Tr}|.\]
If $\mathrm{Tr}(x^{q+1})+L(x)$ is a permutation polynomial of $\mathbb F_{q^n}$, then the left side of the above equation is $|\ker\mathrm{Tr}|=q^{n-1}$, while $|\ker\mathrm{Tr}\cap\ker L^\prime|>0$, which implies $|\ker\mathrm{Tr}\cap\ker L^\prime|=1$ and $|\ker T\setminus\ker\mathrm{Tr}|=0$.

Suppose first that $L_i(1)\notin\mathbb F_q$ for some $i$. In this case,
\[T(x^q+x)=\sum_{i=0}^{m-1}\mathrm{Tr}(L_i(1)(x^q+x))^{2^{m-i}}=\sum_{i=0}^{m-1}\mathrm{Tr}\big(\big(L_i(1)^{q^{n-1}}+L_i(1)\big)x\big)^{2^{m-i}},\]
where at least one summand on the right side is nonzero, and thus the sum as a polynomial over $\mathbb F_{q^n}$ of degree less than $q^n$ is nonzero (look at the degree of each term). Then $T(\ker\mathrm{Tr})$ is nontrivial, and $0<|\ker T\cap\ker\mathrm{Tr}|<|\ker\mathrm{Tr}|=q^{n-1}$. Since $\im T\subseteq\mathbb F_q$, we see that $|\ker T|\ge q^{n-1}$ and
\[|\ker T\setminus\ker\mathrm{Tr}|=|\ker T|-|\ker T\cap\ker\mathrm{Tr}|>0.\]
In this case, $\mathrm{Tr}(x^{q+1})+L(x)$ is not a permutation polynomial of $\mathbb F_{q^n}$.

Suppose now that $L_i(1)\in\mathbb F_q$ for $0\le i<m$, so that
\[T(x)=\sum_{i=0}^{m-1}(L_i(1)\mathrm{Tr}(x))^{2^{m-i}}+\mathrm{Tr}(x)^{2^{m-1}}=\ell(\mathrm{Tr}(x)).\]
If $\ell$ has a nonzero root in $\mathbb F_q$, then clearly $|\ker T\setminus\ker\mathrm{Tr}|>0$. Otherwise,
\[\sum_{u\in\mathbb F_{q^n}}\mathcal S(\mathrm{Tr}(u),t\mathrm{Tr}(u)+L^\prime(u))=|\ker\mathrm{Tr}\cap\ker L^\prime|q^n\]
for every $t\in\ker\mathrm{Tr}$, which completes the proof.
\end{proof}

\begin{remark}
Assume that $\ell$ is a permutation polynomial of $\mathbb F_q$. Then
\[|\ker\mathrm{Tr}\cap\ker L^\prime|=q^{-n}\sum_{u\in\mathbb F_{q^n}}\mathcal S(\mathrm{Tr}(u),t\mathrm{Tr}(u)+L^\prime(u))\]
whenever $t\in\ker\mathrm{Tr}$. Recall that it is the number of roots of
\[\mathrm{Tr}((x+\beta)^{q+1})+L(x+\beta)+\mathrm{Tr}(\beta^{q+1})+L(\beta)\]
in $\mathbb F_{q^n}$ for arbitrary $\beta\in\mathbb F_{q^n}$. Hence, the map $\mathrm{Tr}(x^{q+1})+L(x)$ is $N$-to-$1$ on $\mathbb F_{q^n}$ with $N=|\ker\mathrm{Tr}\cap\ker L^\prime|$; i.e., the inverse image of every element in $\mathbb F_{q^n}$ consists of either $0$ or $N$ elements. Such maps are also of great interest in terms of applications of finite fields. This provides an alternative construction of $N$-to-$1$ maps on finite fields, by determining $|\ker\mathrm{Tr}\cap\ker L^\prime|$ for some certain $L$ as below.
\end{remark}

First, we investigate a relatively simple case where $L$ is a monomial or binomial. For $L(x)=ax^{2^k}+bx^{2^l}$ that is nonzero over $\mathbb F_{q^n}$ with nonnegative integers $k,l$, it is well known that $\ker L=\{0\}$ if $a^\frac{q^n-1}{2^d-1}\ne b^\frac{q^n-1}{2^d-1}$, and $\ker L=\alpha\mathbb F_{2^d}$ otherwise, where $\alpha$ is a nonzero element in $\ker L$ and $d=\gcd(l-k,mn)$. To see this, suppose first that $ab\ne0$ and $k<l$. Then $L(\alpha)=0$ for $\alpha\in\mathbb F_{q^n}^*$ if and only if $\big(\alpha^{2^k}\big)^{2^{l-k}-1}=ab^{-1}$; such $\alpha$ exists in $\mathbb F_{q^n}^*$ if and only if
\[(ab^{-1})^\frac{q^n-1}{2^d-1}=1,\]
where $d=\gcd(l-k,mn)$ so that $2^d-1=\gcd(2^{l-k}-1,q^n-1)$. In this case,
\[\ker L=\alpha\mathbb F_{2^{l-k}}\cap\mathbb F_{q^n}=\alpha\mathbb F_{2^d}.\]
One can easily verify that the conclusion is valid even if $ab=0$ or $k\ge l$. This leads to the following results.

\begin{proposition}
Let $L(x)=ax^{2^k}+bx^{2^l}$ be nonzero with $a,b\in\mathbb F_{q^n}$ and nonnegative integers $k,l$. Denote $d=\gcd(l-k,mn)$, $e=\gcd(l-k,m)$ and $r=\frac de$. If $a^\frac{q^n-1}{2^d-1}\ne b^\frac{q^n-1}{2^d-1}$, then $|\ker\mathrm{Tr}\cap\ker L^\prime|=1$. Assuming $a^\frac{q^n-1}{2^d-1}=b^\frac{q^n-1}{2^d-1}$, we have
\[|\ker\mathrm{Tr}\cap\ker L^\prime|=\begin{cases}2^d&\text{if }a+b\ne0\text{ or }a+b=\mathrm{Tr}_r(a^{-1})=0,\\q^{r-1}&\text{otherwise}\end{cases}\]
when $k\equiv l\pmod m$ with $a+b\in\mathbb F_q$, and
\[|\ker\mathrm{Tr}\cap\ker L^\prime|=\begin{cases}2^d&\text{if }a^\frac{q-1}{2^e-1}\ne b^\frac{q-1}{2^e-1}\text{ or }\frac nr\equiv0\pmod2,\\2^{d-e}&\text{otherwise}\end{cases}\]
when $k\not\equiv l\pmod m$ with $a,b\in\mathbb F_q$.
\end{proposition}
\begin{proof}
Since
\[ne\mathbb Z=n((l-k)\mathbb Z+m\mathbb Z)=n(l-k)\mathbb Z+mnZ\subseteq d\mathbb Z,\]
we find that $r$ divides $n$. If $a^\frac{q^n-1}{2^d-1}\ne b^\frac{q^n-1}{2^d-1}$, then $\ker L^\prime=\{0\}$; otherwise, there exists a nonzero element $\gamma\in\mathbb F_{q^n}$ with $L^\prime(\gamma)=0$. Therefore, it remains to consider the latter case, where $\ker L^\prime=\gamma\mathbb F_{2^d}$.

Suppose $k\equiv l\pmod m$ with $a+b\in\mathbb F_q$. If $a+b\ne0$, then
\[\mathrm{Tr}(L^\prime(x))^{2^l}=\mathrm{Tr}\big((ax)^{2^{l-k}}+bx\big)=\mathrm{Tr}((a+b)x)=(a+b)\mathrm{Tr}(x),\]
which means $\ker L^\prime\subseteq\ker\mathrm{Tr}$, and $|\ker\mathrm{Tr}\cap\ker L^\prime|=|\ker L^\prime|=2^d$. If $a+b=0$, then $\ker L^\prime=a^{-1}\mathbb F_{2^d}=a^{-1}\mathbb F_{q^r}$. Note that $\im\mathrm{Tr}_r=\mathbb F_{q^r}$, and $\mathrm{Tr}(a^{-1}\mathrm{Tr}_r(x))=\mathrm{Tr}(\mathrm{Tr}_r(a^{-1})x)$. If $\mathrm{Tr}_r(a^{-1})=0$, then $\mathrm{Tr}(a^{-1}\im\mathrm{Tr}_r)=\{0\}$ and $a^{-1}\mathbb F_{q^r}\subseteq\ker\mathrm{Tr}$. Otherwise, there exists some $\beta\in\mathbb F_{q^r}^*$ such that $\mathrm{Tr}(a^{-1}\beta)\ne0$, and
\[\mathbb F_q=\mathrm{Tr}(a^{-1}\beta)\mathbb F_q=\mathrm{Tr}(a^{-1}\beta\mathbb F_q)\subseteq\mathrm{Tr}(a^{-1}\mathbb F_{q^r})\subseteq\mathbb F_q;\]
this implies
\[|\ker\mathrm{Tr}\cap a^{-1}\mathbb F_{q^r}|=\frac{|a^{-1}\mathbb F_{q^r}|}{|\mathrm{Tr}(a^{-1}\mathbb F_{q^r})|}=q^{r-1}.\]

Suppose $k\not\equiv l\pmod m$ with $a,b\in\mathbb F_q$. If $a^\frac{q-1}{2^e-1}\ne b^\frac{q-1}{2^e-1}$, then $\ker L^\prime\cap\mathbb F_q=\{0\}$, which implies that $\ker L^\prime\subseteq\ker\mathrm{Tr}$ since
\[\mathrm{Tr}(L^\prime(x))=(a\mathrm{Tr}(x))^{2^{-k}}+(b\mathrm{Tr}(x))^{2^{-l}}=L^\prime(\mathrm{Tr}(x)).\]
Now assuming $a^\frac{q-1}{2^e-1}=b^\frac{q-1}{2^e-1}$, we have $\ker L^\prime=\gamma\mathbb F_{2^d}$ for some $\gamma\in\mathbb F_q^*$. Observe that $\mathrm{Tr}(\gamma\mathbb F_{2^d})=\gamma\mathrm{Tr}(\mathbb F_{2^d})$, so $\ker\mathrm{Tr}\cap\ker L^\prime=\gamma(\ker\mathrm{Tr}\cap\mathbb F_{2^d})$. The additive order of $m$ modulo $d$ is
\[\frac d{\gcd(m,d)}=\frac d{\gcd(m,l-k,mn)}=\frac de,\]
and $x^q=x^{2^m}$ as an automorphism of $\mathbb F_{2^d}$ has order $r=\frac de$ with fixed field $\mathbb F_{2^e}$.  Then the map $\mathrm{Tr}$ restricted to $\mathbb F_{2^d}$ becomes
\[\frac nr\big(x+x^q+\dots+x^{q^{r-1}}\big)=\frac nr\big(x+x^{2^e}+\dots+x^{2^{d-e}}\big),\]
representing the zero map if $\frac nr$ is even, and the trace map of $\mathbb F_{2^d}$ over $\mathbb F_{2^e}$ otherwise. The proof is then complete.
\end{proof}

\begin{corollary}
Let $n$ be odd, $L(x)=ax^{2^k}+bx^{2^l}$ be nonzero with $a,b\in\mathbb F_{q^n}$ and $k,l$ be integers with $0\le k,l<mn$. Denote $d=\gcd(l-k,mn)$ and $s=\gcd(k-1,m)$. When $k\equiv l\pmod m$, $\mathrm{Tr}(x^{q+1})+L(x)$ is a permutation polynomial of $\mathbb F_{q^n}$ if and only if
\begin{itemize}
\item $a+b\in\mathbb F_q^*$, $(a+b)^\frac{q-1}{2^s-1}\ne1$ and $a^\frac{q^n-1}{2^d-1}\ne b^\frac{q^n-1}{2^d-1}$, or
\item $a+b=0$, $d=m$ and $\mathrm{Tr}(a^{-1})\ne0$.
\end{itemize}
When $k\not\equiv l\pmod m$ and $l\equiv1\pmod m$, it is a permutation polynomial of $\mathbb F_{q^n}$ if and only if
\begin{itemize}
\item $a,b\in\mathbb F_q$, $a^\frac{q-1}{2^s-1}\ne(b+1)^\frac{q-1}{2^s-1}$, and
\item either $a^\frac{q^n-1}{2^d-1}\ne b^\frac{q^n-1}{2^d-1}$ or $d$ divides $m$ and $a^\frac{q-1}{2^d-1}=b^\frac{q-1}{2^d-1}$.
\end{itemize} 
In particular, in the case that $b=0$, it is a permutation polynomial of $\mathbb F_{q^n}$ if and only if $a\in\mathbb F_q^*$ and $a^\frac{q-1}{2^{s-1}}\ne1$.
\end{corollary}
\begin{proof}
With the last proposition, it remains to determine under what condition the corresponding polynomial $\ell$ from Theorem \ref{t} is a permutation polynomial of $\mathbb F_q$. If $k\equiv l\pmod m$, then $\ell(x)=((a+b)x)^{2^{-k}}+x^{2^{-1}}$, which permutes $\mathbb F_q$ if and only if $a+b\in\mathbb F_q$ and $(a+b)^\frac{q-1}{2^s-1}\ne1$. If $k\not\equiv l\pmod m$ and $l\equiv1\pmod m$, then $\ell(x)=(ax)^{2^{-k}}+((b+1)x)^{2^{-1}}$, which permutes $\mathbb F_q$ if and only if $a,b\in\mathbb F_q$ and $a^\frac{q-1}{2^s-1}\ne(b+1)^\frac{q-1}{2^s-1}$.
\end{proof}

Consider a polynomial $L(x)=ax+bx^q+cx^{q^2}$ over $\mathbb F_{q^n}$ with $n=3$. It follows from {\cite[Proposition 4.4]{wu2013linearized}} that $|\ker L|=q^{n-r}$, where $r$ is the rank of the matrix
\[\begin{pmatrix}a&b&c\\c^q&a^q&b^q\\b^{q^2}&c^{q^2}&a^{q^2}\end{pmatrix}\]
with determinant $a^{q^2+q+1}+b^{q^2+q+1}+c^{q^2+q+1}+\mathrm{Tr}\big(ab^qc^{q^2}\big)$. This allows us to get $|\ker\mathrm{Tr}\cap\ker L^\prime|$ directly from the coefficients of $L$. If $L(1)\in\mathbb F_q^*$, then
\[\mathrm{Tr}(L^\prime(x))=\mathrm{Tr}(L(1)x)=L(1)\mathrm{Tr}(x),\]
and thus $\ker L^\prime\subseteq\ker\mathrm{Tr}$ and $\ker\mathrm{Tr}\cap\ker L^\prime=\ker L^\prime$. Another special case is discussed below.

\begin{proposition}
Let $n=3$ and $L(x)=ax+bx^q+cx^{q^2}$ be a polynomial over $\mathbb F_{q^n}$ with $L(1)=0$. Then
\[|\ker\mathrm{Tr}\cap\ker L^\prime|=\begin{cases}1&\text{if }\mathrm{Tr}(a^{q+1}+a^qb+b^{q+1})\ne0,\\q^2&\text{if }a+b^{q^2}=a+c^q=0,\\q&\text{otherwise}.\end{cases}\]
\end{proposition}
\begin{proof}
Consider
\[L^\prime(x^q+x)=\big(a+b^{q^2}\big)x+(a+c^q)x^q+\big(b^{q^2}+c^q\big)x^{q^2},\]
whose kernel has cardinality
\[\frac{q^3}{L^\prime(\ker\mathrm{Tr})}=q^3\frac{|\ker\mathrm{Tr}\cap\ker L^\prime|}{|\ker\mathrm{Tr}|}=q|\ker\mathrm{Tr}\cap\ker L^\prime|.\]
Its corresponding matrix
\[\begin{pmatrix}a+b^{q^2}&a+c^q&b^{q^2}+c^q\\b+c^{q^2}&a^q+b&a^q+c^{q^2}\\a^{q^2}+c&b^q+c&a^{q^2}+b^q\end{pmatrix}\]
is equivalent to
\[\begin{pmatrix}a+b^{q^2}&a+c^q&0\\b+c^{q^2}&a^q+b&0\\0&0&0\end{pmatrix}\]
by eliminating the last row and the last column. A direct calculation yields
\[\begin{split}&\mathrel{\phantom{=}}\begin{vmatrix}a+b^{q^2}&a+c^q\\b+c^{q^2}&a^q+b\end{vmatrix}\\&=a^{q+1}+ab+a^qb^{q^2}+b^{q^2+1}+ab+ac^{q^2}+bc^q+c^{q^2+q}\\&=\mathrm{Tr}(a^{q+1}+a^qb+b^{q+1}),\end{split}\]
and the desired result follows.
\end{proof}

\begin{corollary}
Let $n=3$ and $L(x)=\big(ax+bx^q+cx^{q^2}\big)^{2^k}$ be a polynomial over $\mathbb F_{q^n}$ for an integer $k$ with $0\le k<m$. Then $\mathrm{Tr}(x^{q+1})+L(x)$ is a permutation polynomial of $\mathbb F_{q^n}$ if and only if $(a+b+c)^\frac{q-1}{2^s-1}\ne1$ with $s=\gcd(k-1,m)$, and either $a+b+c\in\mathbb F_q^*$ and $a^{q^2+q+1}+b^{q^2+q+1}+c^{q^2+q+1}+\mathrm{Tr}\big(ab^qc^{q^2}\big)\ne0$, or $a+b+c=0$ and $\mathrm{Tr}(a^{q+1}+a^qb+b^{q+1})\ne0$.
\end{corollary}

Having studied the polynomial $\mathrm{Tr}(x^{q+1})+L(x)$ in terms of the coefficients of $L$, we will now approach it from another perspective. The conditions in Theorem \ref{t} can be restated in the following way.

\begin{lemma}\label{l}
Let
\[L(x)=\sum_{i=0}^{m-1}L_i\big(x^{2^i}\big),\]
where $L_i$ is a $q$-linear polynomial over $\mathbb F_{q^n}$. Then
\begin{enumerate}
\item[\upshape(i)]$L^\prime(\ker\mathrm{Tr})\subseteq\ker\mathrm{Tr}$ if and only if $L_i(1)\in\mathbb F_q$ for $0\le i<m$, in which case there exists a unique map $l$ on $\mathbb F_q$ such that $\mathrm{Tr}(L^\prime(x))=l(\mathrm{Tr}(x))$, written as
\[l(x)=\sum_{i=0}^{m-1}(L_i(1)x)^{2^{m-i}};\]
\item[\upshape(ii)]if $L^\prime(\ker\mathrm{Tr})\subseteq\ker\mathrm{Tr}$, then $L^\prime$ defines an endomorphism of $\mathbb F_{q^n}/\ker\mathrm{Tr}$ in the obvious way:
\[\ker\mathrm{Tr}+\alpha\mapsto\ker\mathrm{Tr}+L^\prime(\alpha)\qquad(\alpha\in\mathbb F_{q^n}),\]
and it is an automorphism if and only if the above map $l$ is a linear automorphism of $\mathbb F_q/\mathbb F_2$;
\item[\upshape(iii)]$|L^\prime(\ker\mathrm{Tr})|=|\ker\mathrm{Tr}|$ if and only if $\ker\mathrm{Tr}\cap\ker L^\prime=\{0\}$.
\end{enumerate}
\end{lemma}
\begin{proof}
If $L_i(1)\in\mathbb F_q$ for $0\le i<m$, then apparently $L^\prime(\ker\mathrm{Tr})\subseteq\ker\mathrm{Tr}$ as
\[\mathrm{Tr}(L^\prime(x))=\sum_{i=0}^{m-1}\mathrm{Tr}(L_i(1)x)^{2^{m-i}}=\sum_{i=0}^{m-1}(L_i(1)\mathrm{Tr}(x))^{2^{m-i}}.\]
Conversely, if $L^\prime(\ker\mathrm{Tr})\subseteq\ker\mathrm{Tr}$, then $\mathrm{Tr}(L^\prime(x^q+x))=0$, while
\[\mathrm{Tr}(L^\prime(x^q+x))=\sum_{i=0}^{m-1}\mathrm{Tr}(L_i(1)(x^q+x))^{2^{m-i}}=\sum_{i=0}^{m-1}\mathrm{Tr}\big(\big(L_i(1)^{q^{n-1}}+L_i(1)\big)x\big)^{2^{m-i}},\]
where the right side formally has degree less than $q^n$, and thus $L_i(1)+L_i(1)^q=0$ for $0\le i<m$. The existence and uniqueness of the map $l$ can be easily verified.

Suppose now that $L^\prime(\ker\mathrm{Tr})\subseteq\ker\mathrm{Tr}$. Denote by $\tau$ the canonical isomorphism from $\mathbb F_{q^n}/\ker\mathrm{Tr}$ to $\mathbb F_q$, so that $\tau(\ker\mathrm{Tr}+\alpha)=\mathrm{Tr}(\alpha)$ and
\[\tau(\ker\mathrm{Tr}+L^\prime(\alpha))=\mathrm{Tr}(L^\prime(\alpha))=l(\mathrm{Tr}(\alpha))=(l\circ \tau)(\ker\mathrm{Tr}+\alpha)\]
for arbitrary $\alpha\in\mathbb F_{q^n}$. Accordingly, the map in (ii) is actually $\tau^{-1}\circ l\circ\tau$, well-defined and clearly an endomorphism of $\mathbb F_{q^n}/\ker\mathrm{Tr}$. This proves (ii).

The third statement follows from
\[|\ker\mathrm{Tr}|=|\ker\mathrm{Tr}\cap\ker L^\prime||L^\prime(\ker\mathrm{Tr})|,\]
with $L^\prime$ regarded as a homomorphism from $\ker\mathrm{Tr}$ onto $L^\prime(\ker\mathrm{Tr})$.
\end{proof}

\begin{corollary}\label{c}
Let $n$ be odd. Then $\mathrm{Tr}(x^{q+1})+L(x)$ is a permutation polynomial of $\mathbb F_{q^n}$ if and only if $L^\prime(\ker\mathrm{Tr})=\ker\mathrm{Tr}$ and $L^\prime(x)+x^\frac12$ acts as an automorphism of $\mathbb F_{q^n}/\ker\mathrm{Tr}$.
\end{corollary}

With the established results, one can construct more permutation polynomials using $2$-linear polynomials with different properties. In what follows, let $n$ be odd and $\lambda$ be a $2$-linear polynomial over $\mathbb F_{q^n}$.

\begin{proposition}
If $\lambda^\prime(\ker\mathrm{Tr})=\ker\mathrm{Tr}$, $\mathrm{Tr}(\lambda^\prime(x))=\mathrm{Tr}(x)^\frac12$ and $\ell$ is a $2$-linear polynomial of degree less than $q$ with coefficients in $\mathbb F_q$ that permutes $\mathbb F_q$, then $\mathrm{Tr}(x^{q+1}+\ell(x))+\lambda(x)$ is a permutation polynomial of $\mathbb F_{q^n}$.
\end{proposition}
\begin{proof}
For $\alpha\in\mathbb F_{q^n}$, if $\mathrm{Tr}(\alpha)=0$, then
\[\ell^\prime(\mathrm{Tr}(\alpha))+\lambda^\prime(\alpha)=\lambda^\prime(\alpha),\]
which means $\ell^\prime(\mathrm{Tr}(x))+\lambda^\prime(x)$ maps $\ker\mathrm{Tr}$ onto $\ker\mathrm{Tr}$, and then $\ell^\prime(\mathrm{Tr}(x))+\lambda^\prime(x)+x^\frac12$ defines an endomorphism of $\mathbb F_{q^n}/\ker\mathrm{Tr}$. Furthermore,
\[\ker\mathrm{Tr}+\ell^\prime(\mathrm{Tr}(\alpha))+\lambda^\prime(\alpha)+\alpha^\frac12=\ker\mathrm{Tr}+\ell^\prime(\mathrm{Tr}(\alpha)),\]
since $\mathrm{Tr}\big(\lambda^\prime(\alpha)+\alpha^\frac12\big)=0$. If $\ell^\prime(\mathrm{Tr}(\alpha))\in\ker\mathrm{Tr}$, then $\ell^\prime(\mathrm{Tr}(\alpha))=\mathrm{Tr}(\ell^\prime(\mathrm{Tr}(\alpha)))=0$ as $\ell^\prime$ maps $\mathbb F_q$ onto $\mathbb F_q$ ($\ell^\prime$ is, by definition, exactly the adjoint of $\ell$ when they are viewed as linear endomorphisms of $\mathbb F_q/\mathbb F_2$), and thus $\mathrm{Tr}(\alpha)=0$. This implies that the kernel of $\ell^\prime(\mathrm{Tr}(x))+\lambda^\prime(x)+x^\frac12$ as an endomorphism of $\mathbb F_{q^n}/\ker\mathrm{Tr}$ is trivial.
\end{proof}

\begin{proposition}
If $L^\prime(\ker\mathrm{Tr})=\lambda^\prime(\ker\mathrm{Tr})=\ker\mathrm{Tr}$ and $\mathrm{Tr}(\lambda^\prime(x))=0$, then both $\mathrm{Tr}(x^{q+1})+L(\lambda(x))$ and $\mathrm{Tr}(x^{q+1})+\lambda(L(x))$ are permutation polynomials of $\mathbb F_{q^n}$.
\end{proposition}
\begin{proof}
This follows from the fact that $\im\lambda^\prime\subseteq\ker\mathrm{Tr}$ and $\lambda^\prime$ induces a trivial endomorphism of $\mathbb F_{q^n}/\ker\mathrm{Tr}$.
\end{proof}


Finally, we can derive more from some already obtained permutation polynomial of the form $\mathrm{Tr}(x^{q+1})+L(x)$.

\begin{theorem}
Assume that $\mathrm{Tr}(x^{q+1})+L(x)$ is a permutation polynomial of $\mathbb F_{q^n}$. Then $|\ker L|\le q$ and
\begin{enumerate}
\item[\upshape(i)]both
\[\mathrm{Tr}(x^{q+1})+L(\lambda(x))\quad\text{and}\quad\mathrm{Tr}(x^{q+1})+\lambda(L(x))\]
are permutation polynomials of $\mathbb F_{q^n}$ if $\lambda^\prime(\ker\mathrm{Tr})=\ker\mathrm{Tr}$ and $\mathrm{Tr}(\lambda^\prime(x))=\mathrm{Tr}(x)$;
\item[\upshape(ii)]both
\[\mathrm{Tr}(x^{q+1})+L(x)+L(\mathrm{Tr}(\lambda(x)))+\mathrm{Tr}(\lambda(x))^2\]
and
\[\mathrm{Tr}(x^{q+1})+L(x)+\mathrm{Tr}(\lambda(L(x)))+\mathrm{Tr}(\lambda(x^2))\]
are permutation polynomials of $\mathbb F_{q^n}$ if $x+\mathrm{Tr}(\lambda^\prime(x))$ is a permutation polynomial of $\mathbb F_q$.
\end{enumerate}
\end{theorem}
\begin{proof}
By the assumption, one has $\ker\mathrm{Tr}\cap\ker L^\prime=\{0\}$, and then
\[|\ker L^\prime|=|\ker\mathrm{Tr}\cap\ker L^\prime||\mathrm{Tr}(\ker L^\prime)|=|\mathrm{Tr}(\ker L^\prime)|\le q.\]
For (i), note that
\[\lambda^\prime(L^\prime(\ker\mathrm{Tr}))=L^\prime(\lambda^\prime(\ker\mathrm{Tr}))=\ker\mathrm{Tr}.\]
Furthermore, $\lambda^\prime$ fixes every coset in $\mathbb F_{q^n}/\ker\mathrm{Tr}$, so $\lambda^\prime(L^\prime(x))+x^\frac12$ and $L^\prime(\lambda^\prime(x))+x^\frac12$ induce the same automorphism of $\mathbb F_{q^n}/\ker\mathrm{Tr}$ as $L^\prime(x)+x^\frac12$ does. Hence, the polynomials are both permutation polynomials of $\mathbb F_{q^n}$ by Corollary \ref{c}.

To prove (ii), let $\varphi(x)=x+\lambda^\prime(\mathrm{Tr}(x))$, which defines an automorphism of $\mathbb F_{q^n}/\ker\mathrm{Tr}$ according to Lemma \ref{l}. Consider the adjoint of $L(x)+L(\mathrm{Tr}(\lambda(x)))+\mathrm{Tr}(\lambda(x))^2$:
\[L^\prime(x)+\lambda^\prime(\mathrm{Tr}(L^\prime(x)))+\lambda^\prime\big(\mathrm{Tr}\big(x^\frac12\big)\big)=\varphi\big(L^\prime(x)+x^\frac12\big)+x^\frac12.\]
Clearly it maps $\ker\mathrm{Tr}$ onto $\ker\mathrm{Tr}$, and
\[L^\prime(x)+\lambda^\prime(\mathrm{Tr}(L^\prime(x)))+\lambda^\prime\big(\mathrm{Tr}\big(x^\frac12\big)\big)+x^\frac12=\varphi\big(L^\prime(x)+x^\frac12\big),\]
which permutes $\mathbb F_{q^n}/\ker\mathrm{Tr}$. For $L(x)+\mathrm{Tr}(\lambda(L(x)))+\mathrm{Tr}(\lambda(x^2))$, we have
\[L^\prime(x)+L^\prime(\lambda^\prime(\mathrm{Tr}(x)))+\lambda^\prime(\mathrm{Tr}(x))^\frac12=L^\prime(\varphi(x))+\varphi(x)^\frac12+x^\frac12,\]
and arrive at the conclusion by the same argument.
\end{proof}

\section{The case of even $n$}\label{even}

In this section, let $n$ be even and consider $\mathrm{Tr}(Ax^{q+1})+L(x)$ with $A\in\mathbb F_{q^n}^*$. As in the case of odd $n$, it is a permutation polynomial of $\mathbb F_{q^n}$ if and only if for every $\beta\in\mathbb F_{q^n}$,
\[\mathrm{Tr}(Ax^{q+1})+\mathrm{Tr}\big(\big(A^{q^{n-1}}\beta^{q^{n-1}}+A\beta^q\big)x\big)+L(x)\]
has only one root in $\mathbb F_{q^n}$. If $A=\alpha^{q+1}$ for some $\alpha\in\mathbb F_{q^n}^*$, then it can be reduced to the case $A=1$, where $x^{q^{n-1}}+x^q$ maps $\mathbb F_{q^n}$ onto $\ker\mathrm{Tr}_2$ with kernel $\mathbb F_{q^2}$. Otherwise (i.e., $A^\frac{q^n-1}{q+1}\ne1$), $A^{q^{n-1}}x^{q^{n-1}}+Ax^q$ induces a linear automorphism of $\mathbb F_{q^n}/\mathbb F_q$ by a simple investigation. Then $\mathrm{Tr}(Ax^{q+1})+L(x)$ is a permutation polynomial of $\mathbb F_{q^n}$ if and only if
\[q^{-n}\sum_{u\in\mathbb F_{q^n}}\mathcal S(A\mathrm{Tr}(u),t\mathrm{Tr}(u)+L^\prime(u))=1\]
for every $t\in\ker\mathrm{Tr}_2$ when $A=1$, and for every $t\in\mathbb F_{q^n}$ when $A^\frac{q^n-1}{q+1}\ne1$. For our purpose, the character sum will be evaluated using the following lemma.

\begin{lemma}[{\cite[Theorem 5.2, Theorem 5.3]{coulter1999evaluation}}]
Suppose that $n$ is even. For $a\in\mathbb F_q^*$ and $b\in\mathbb F_{q^n}$, if $b=\beta^{q^2}+\beta$ for some $\beta\in\mathbb F_{q^n}$, then $\mathcal S(a,b)=\chi(a^{-1}\beta^{q+1})(-q)^{\frac n2+1}$; otherwise, $\mathcal S(a,b)=0$. For $a\in\mathbb F_{q^n}^*$ such that $a^\frac{q^n-1}{q+1}\ne1$ and $b\in\mathbb F_{q^n}$, there exists a unique element $\beta$ in $\mathbb F_{q^n}$ satisfying $a^{q^{n-1}}\beta^{q^{n-1}}+a\beta^q=b$ and $\mathcal S(a,b)=\chi(a\beta^{q+1})(-q)^\frac n2$.
\end{lemma}

\begin{theorem}
Let $A\in\mathbb F_{q^n}^*$. Then $\mathrm{Tr}(x^{q+1})+L(x)$ is a permutation polynomial of $\mathbb F_{q^n}$ if and only if $\ker(\mathrm{Tr}_2\circ L^\prime)\subseteq\ker\mathrm{Tr}$ and $\ker L=\{0\}$. If $A^\frac{q^n-1}{q+1}\ne1$, then $\mathrm{Tr}(Ax^{q+1})+L(x)$ is not a permutation polynomial of $\mathbb F_{q^n}$.
\end{theorem}
\begin{proof}
First, we study $\mathrm{Tr}(x^{q+1})+L(x)$ and
\[\sum_{u\in\mathbb F_{q^n}}\mathcal S(\mathrm{Tr}(u),t\mathrm{Tr}(u)+L^\prime(u))\]
for $t\in\ker\mathrm{Tr}_2$. It is easy to see that
\[\sum_{u\in\ker\mathrm{Tr}}\mathcal S(\mathrm{Tr}(u),t\mathrm{Tr}(u)+L^\prime(u))=|\ker\mathrm{Tr}\cap\ker L^\prime|q^n.\]
Also, for $u\in\mathbb F_{q^n}$ such that $\mathrm{Tr}(u)\ne0$, $\mathcal S(\mathrm{Tr}(u),t\mathrm{Tr}(u)+L^\prime(u))$ is nonzero if and only if $\mathrm{Tr}_2(t\mathrm{Tr}(u)+L^\prime(u))=0$; that is, $u\in\ker T$ for $T=\mathrm{Tr}_2\circ L^\prime$. Noticing that
\[\begin{split}&\mathrel{\phantom{=}}\sum_{t\in\ker\mathrm{Tr}_2}\sum_{u\in\ker T\setminus\ker\mathrm{Tr}}\mathcal S(\mathrm{Tr}(u),t\mathrm{Tr}(u)+L^\prime(u))\\&=\sum_{u\in\ker T\setminus\ker\mathrm{Tr}}\sum_{t\in\ker\mathrm{Tr}_2}\mathcal S(\mathrm{Tr}(u),t\mathrm{Tr}(u)+L^\prime(u))\\&=\sum_{u\in\ker T\setminus\ker\mathrm{Tr}}q^{-2}\sum_{\beta\in\mathbb F_{q^n}}\mathcal S(\mathrm{Tr}(u),\beta^{q^2}+\beta)\\&=\sum_{u\in\ker T\setminus\ker\mathrm{Tr}}q^{-2}\sum_{\beta\in\mathbb F_{q^n}}\chi(\mathrm{Tr}(u)^{-1}\beta^{q+1})(-q)^{\frac n2+1}\\&=\sum_{u\in\ker T\setminus\ker\mathrm{Tr}}q^{-2}\mathcal S(\mathrm{Tr}(u)^{-1},0)(-q)^{\frac n2+1}\\&=|\ker T\setminus\ker\mathrm{Tr}|q^n,\end{split}\]
and that
\[\sum_{t\in\ker\mathrm{Tr}_2}\sum_{u\in\ker\mathrm{Tr}}\mathcal S(\mathrm{Tr}(u),t\mathrm{Tr}(u)+L^\prime(u))=|\ker\mathrm{Tr}\cap\ker L^\prime|q^{2n-2},\]
we obtain
\[\sum_{t\in\ker\mathrm{Tr}_2}\sum_{u\in\mathbb F_{q^n}}\mathcal S(\mathrm{Tr}(u),t\mathrm{Tr}(u)+L^\prime(u))=|\ker\mathrm{Tr}\cap\ker L^\prime|q^{2n-2}+|\ker T\setminus\ker\mathrm{Tr}|q^n.\]
Then $\mathrm{Tr}(x^{q+1})+L(x)$ is a permutation polynomial of $\mathbb F_{q^n}$ only if
\[|\ker\mathrm{Tr}\cap\ker L^\prime|q^{2n-2}+|\ker T\setminus\ker\mathrm{Tr}|q^n=q^{2n-2};\]
that is, $|\ker\mathrm{Tr}\cap\ker L^\prime|=1$ and $|\ker T\setminus\ker\mathrm{Tr}|=0$. Necessarily, we get
\[\ker L^\prime\subseteq\ker T\subseteq\ker\mathrm{Tr},\]
and then $\ker L^\prime=\{0\}$. Conversely, if $\ker T\subseteq\ker\mathrm{Tr}$ and $\ker L^\prime=\{0\}$, then $|\ker T\setminus\ker\mathrm{Tr}|=0$ and
\[\sum_{u\in\mathbb F_{q^n}}\mathcal S(\mathrm{Tr}(u),t\mathrm{Tr}(u)+L^\prime(u))=|\ker\mathrm{Tr}\cap\ker L^\prime|q^n=q^n\]
for every $t\in\ker\mathrm{Tr}_2$.

Now consider $\mathrm{Tr}(Ax^{q+1})+L(x)$ for $A\in\mathbb F_{q^n}^*$ with $A^\frac{q^n-1}{q+1}\ne1$. Assuming that it is a permutation polynomial of $\mathbb F_{q^n}$, we get
\[\sum_{t\in\mathbb F_{q^n}}\sum_{u\in\mathbb F_{q^n}}\mathcal S(A\mathrm{Tr}(u),t\mathrm{Tr}(u)+L^\prime(u))=q^{2n}.\]
As before,
\[\sum_{t\in\mathbb F_{q^n}}\sum_{u\in\ker\mathrm{Tr}}\mathcal S(A\mathrm{Tr}(u),t\mathrm{Tr}(u)+L^\prime(u))=|\ker\mathrm{Tr}\cap\ker L^\prime|q^{2n}.\]
For arbitrary $u\in\mathbb F_{q^n}\setminus\ker\mathrm{Tr}$, clearly $(A\mathrm{Tr}(u))^{q^{n-1}}x^{q^{n-1}}+A\mathrm{Tr}(u)x^q$ is a permutation polynomial of $\mathbb F_{q^n}$ since $(A\mathrm{Tr}(u))^\frac{q^n-1}{q+1}=A^\frac{q^n-1}{q+1}\ne1$, and $t\mathrm{Tr}(u)+L^\prime(u)$ runs through $\mathbb F_{q^n}$ as $t$ does, so
\[\begin{split}&\mathrel{\phantom{=}}\sum_{t\in\mathbb F_{q^n}}\mathcal S(A\mathrm{Tr}(u),t\mathrm{Tr}(u)+L^\prime(u))\\&=\sum_{\beta\in\mathbb F_{q^n}}\mathcal S\big(A\mathrm{Tr}(u),(A\mathrm{Tr}(u))^{q^{n-1}}\beta^{q^{n-1}}+A\mathrm{Tr}(u)\beta^q\big)\\&=\sum_{\beta\in\mathbb F_{q^n}}\chi(A\mathrm{Tr}(u)\beta^{q+1})(-q)^\frac n2\\&=\mathcal S(A\mathrm{Tr}(u),0)(-q)^\frac n2=q^n.\end{split}\]
Therefore,
\[\sum_{t\in\mathbb F_{q^n}}\sum_{u\in\mathbb F_{q^n}\setminus\ker\mathrm{Tr}}\mathcal S(A\mathrm{Tr}(u),t\mathrm{Tr}(u)+L^\prime(u))=\sum_{u\in\mathbb F_{q^n}\setminus\ker\mathrm{Tr}}q^n=q^{2n}-q^{2n-1},\]
while
\[\sum_{t\in\mathbb F_{q^n}}\sum_{u\in\mathbb F_{q^n}}\mathcal S(A\mathrm{Tr}(u),t\mathrm{Tr}(u)+L^\prime(u))=|\ker\mathrm{Tr}\cap\ker L^\prime|q^{2n}+q^{2n}-q^{2n-1}\]
is not divisible by $q^{2n}$, a contradiction.
\end{proof}

\begin{remark}
By the same argument, if $\ker(\mathrm{Tr}_2\circ L^\prime)\subseteq\ker\mathrm{Tr}$, then $\mathrm{Tr}(x^{q+1})+L(x)$ is an $N$-to-$1$ map on $\mathbb F_{q^n}$, where $N=|\ker L^\prime|$.
\end{remark}

For explicit constructions of such permutation polynomials, we focus on the case where $L$ is a monomial or binomial.

\begin{proposition}
Let $L(x)=L_k\big(x^{2^k}\big)+L_l\big(x^{2^l}\big)$ for $q^2$-linear polynomial $L_k$ and $L_l$, $k,l$ be integers with $0\le k<l<2m$ and $e=\gcd(l-k,2m)$. Denote $a=L_k(1)$, $b=L_l(1)$ and $\delta=a^{q^2}b+ab^{q^2}$. Then $\ker(\mathrm{Tr}_2\circ L^\prime)\subseteq\ker\mathrm{Tr}$ if and only if
\begin{itemize}
\item $l-k=m$, $a\notin\mathbb F_{q^2}$, $\delta\ne0$ and
\[\delta^q\big(a^{q^2}+a\big)+\delta\big(b^{q^3}+b^q\big)=\delta^q\big(b^{q^2}+b\big)+\delta\big(a^{q^3}+a^q\big)=0,\]
or
\item $a,b\in\mathbb F_{q^2}$ and $a^\frac{q^2-1}{2^e-1}\ne b^\frac{q^2-1}{2^e-1}$, or
\item $e$ divides $m$, $a,b\in\mathbb F_{q^2}^*$ and $a^\frac{2^{l-k}(q-1)}{2^e-1}=b^\frac{q-1}{2^e-1}$.
\end{itemize}
\end{proposition}
\begin{proof}
Consider $\mathbb F_{q^n}$ as a vector space over $\mathbb F_{q^2}$. If $1$ is not a linear combination of $a$ and $b$ over $\mathbb F_{q^2}$ (including the case $a=b=0$), then there exists $\alpha\in\mathbb F_{q^n}$ such that $\mathrm{Tr}_2(a\alpha)=\mathrm{Tr}_2(b\alpha)=0$ while $\mathrm{Tr}_2(\alpha)\notin\mathbb F_q$, which means
\[\mathrm{Tr}_2(L^\prime(\alpha))=\mathrm{Tr}_2(a\alpha)^{2^{-k}}+\mathrm{Tr}_2(b\alpha)^{2^{-l}}=0\]
and $\mathrm{Tr}(\alpha)=\mathrm{Tr}_2(\alpha)+\mathrm{Tr}_2(\alpha)^q\ne0$.

Suppose $1=\mu a+\nu b$ for some $\mu,\nu\in\mathbb F_{q^2}$ with $a\notin\mathbb F_{q^2}$. We show that $\ker(\mathrm{Tr}_2\circ L^\prime)\subseteq\ker\mathrm{Tr}$ if and only if $l-k=m$ and $\mu^q=\nu$. Note that $\nu\ne0$ and
\[\begin{split}&\mathrel{\phantom{=}}\nu\mathrm{Tr}_2(L^\prime(x))^{2^l}\\&=\nu\mathrm{Tr}_2(ax)^{2^{l-k}}+\mathrm{Tr}_2(\nu bx)\\&=\nu\mathrm{Tr}_2(ax)^{2^{l-k}}+\mathrm{Tr}_2((1+\mu a)x)\\&=\mu\mathrm{Tr}_2(ax)+\nu\mathrm{Tr}_2(ax)^{2^{l-k}}+\mathrm{Tr}_2(x),\end{split}\]
whose kernel is exactly that of $\mathrm{Tr}_2(L^\prime(x))$. The polynomial $\mu x+\nu x^{2^{l-k}}+\big(\mu x+\nu x^{2^{l-k}}\big)^q$ vanishes on $\mathbb F_{q^2}$ if and only if
\[\mu x+\nu x^{2^{l-k}}+\mu^qx^q+\nu^qx^{2^{l-k+m}}\equiv0\pmod{x^{q^2}+x}\]
in $\mathbb F_{q^2}[x]$; that is, $l-k=m$ and $\mu^q=\nu$, by reducing that to a polynomial of degree less than $q^2$. If this is the case, then
\[\mathrm{Tr}_2(\alpha)=\mu\mathrm{Tr}_2(a\alpha)+\nu\mathrm{Tr}_2(a\alpha)^q\in\mathbb F_q\]
and $\mathrm{Tr}(\alpha)=0$ for all $\alpha\in\mathbb F_{q^n}$ such that $\mathrm{Tr}_2(L^\prime(\alpha))=0$. Otherwise, there exists $\alpha\in\mathbb F_{q^n}$ such that $\mathrm{Tr}_2(a\alpha)$ satisfies
\[\mu\mathrm{Tr}_2(a\alpha)+\nu\mathrm{Tr}_2(a\alpha)^{2^{l-k}}\notin\mathbb F_q,\]
and $\mathrm{Tr}_2(\alpha)=\mu\mathrm{Tr}_2(a\alpha)+\nu\mathrm{Tr}_2(a\alpha)^{2^{l-k}}$, which implies $\mathrm{Tr}_2(L^\prime(\alpha))=0$ while $\mathrm{Tr}(\alpha)\ne0$.

Suppose $1=\mu a+\nu b$ for some $\mu,\nu\in\mathbb F_{q^2}$ with $a\in\mathbb F_{q^2}$, so that one of $a$ and $b$ is nonzero and
\[\mathrm{Tr}_2(L^\prime(x))=(a\mathrm{Tr}_2(x))^{2^{-k}}+\mathrm{Tr}_2(bx)^{2^{-l}}.\]
If $b\notin\mathbb F_{q^2}$, then clearly there exists $\alpha\in\mathbb F_{q^n}$ with $\mathrm{Tr}_2(\alpha)$ arbitrary in $\mathbb F_{q^2}$ and $\mathrm{Tr}_2(b\alpha)=(a\mathrm{Tr}_2(\alpha))^{2^{l-k}}$, yielding an element of $\ker(\mathrm{Tr}_2\circ L^\prime)\setminus\ker\mathrm{Tr}$. Assume in addition that $b\in\mathbb F_{q^2}$. Then
\[\mathrm{Tr}_2(L^\prime(x))^{2^l}=(a\mathrm{Tr}_2(x))^{2^{l-k}}+b\mathrm{Tr}_2(x),\]
which indicates that $\ker(\mathrm{Tr}_2\circ L^\prime)\subseteq\ker\mathrm{Tr}$ if and only if the kernel of $(ax)^{2^{l-k}}+bx$ in $\mathbb F_{q^2}$ is contained in $\mathbb F_q$. If $a^\frac{q^2-1}{2^e-1}\ne b^\frac{q^2-1}{2^e-1}$, then the kernel is $\{0\}$; otherwise, it is $\gamma\mathbb F_{2^e}$ for some $\gamma\in\mathbb F_{q^2}^*$ such that $\gamma^{2^{l-k}-1}=a^{-2^{l-k}}b$. In the latter case, $\gamma\mathbb F_{2^e}\subseteq\mathbb F_q$ if and only if $e$ divides $m$ and $\gamma\in\mathbb F_q$; i.e., $e$ divides $m$ and $\big(a^{-2^{l-k}}b\big)^\frac{q-1}{2^e-1}=1$, as easily seen.

Finally, assuming $a\notin\mathbb F_{q^2}$, we prove that $1=\mu a+\mu^qb$ for some $\mu\in\mathbb F_{q^2}$ if and only if $\delta=a^{q^2}b+ab^{q^2}\ne0$ and
\[\delta^q\big(a^{q^2}+a\big)+\delta\big(b^{q^3}+b^q\big)=\delta^q\big(b^{q^2}+b\big)+\delta\big(a^{q^3}+a^q\big)=0.\]
If $1=\mu a+\mu^qb$ for some $\mu\in\mathbb F_{q^2}$, then
\[\begin{pmatrix}a&b\\a^{q^2}&b^{q^2}\end{pmatrix}\begin{pmatrix}\mu\\\mu^q\end{pmatrix}=\begin{pmatrix}1\\1\end{pmatrix}\]
and $\delta\ne0$ (otherwise, $a^{-1}b\in\mathbb F_{q^2}$ and $\mu a+\mu^qb=(\mu+\mu^qa^{-1}b)a\notin\mathbb F_{q^2}$), so that
\[\begin{pmatrix}\mu\\\mu^q\end{pmatrix}=\begin{pmatrix}a&b\\a^{q^2}&b^{q^2}\end{pmatrix}^{-1}\begin{pmatrix}1\\1\end{pmatrix}=\delta^{-1}\begin{pmatrix}b^{q^2}&b\\a^{q^2}&a\end{pmatrix}\begin{pmatrix}1\\1\end{pmatrix}=\delta^{-1}\begin{pmatrix}b^{q^2}+b\\a^{q^2}+a\end{pmatrix}.\]
This implies
\[\frac{b^{q^3}+b^q}{\delta^q}=\mu^q=\frac{a^{q^2}+a}\delta,\]
and
\[\frac{a^{q^3}+a^q}{\delta^q}=\mu^{q^2}=\mu=\frac{b^{q^2}+b}\delta.\]
The converse is apparent by letting $\mu=\delta^{-1}\big(b^{q^2}+b)$.
\end{proof}

\begin{corollary}
Let $L(x)=ax^{2^k}+bx^{2^l}$ be nonzero with $a,b\in\mathbb F_{q^n}$ and $k,l$ be integers with $0\le k<l<mn$. Denote $d=\gcd(l-k,mn)$, $e=\gcd(l-k,2m)$ and $\delta=a^{q^2}b+ab^{q^2}$. Then $\mathrm{Tr}(x^{q+1})+L(x)$ is a permutation polynomial of $\mathbb F_{q^n}$ if and only if $a^\frac{q^n-1}{2^d-1}\ne b^\frac{q^n-1}{2^d-1}$ and one of the following conditions holds:
\begin{itemize}
\item $l-k\equiv m\pmod{2m}$, $a\notin\mathbb F_{q^2}$, $\delta\ne0$ and
\[\delta^q\big(a^{q^2}+a\big)+\delta\big(b^{q^3}+b^q\big)=\delta^q\big(b^{q^2}+b\big)+\delta\big(a^{q^3}+a^q\big)=0;\]
\item $k\not\equiv l\pmod{2m}$, $a,b\in\mathbb F_{q^2}$ and $a^\frac{q^2-1}{2^e-1}\ne b^\frac{q^2-1}{2^e-1}$;
\item $k\equiv l\pmod{2m}$ and $a+b\in\mathbb F_{q^2}$.
\end{itemize}
\end{corollary}
\begin{proof}
It suffices to check the conditions in the last proposition in the case that $a^\frac{q^n-1}{2^d-1}\ne b^\frac{q^n-1}{2^d-1}$. The first two conditions there are obvious, and the third fails to hold because in that case, the polynomial $(ax)^{2^{l-k}}+bx$ over $\mathbb F_{q^2}$ does not permute $\mathbb F_{q^2}$, and neither does $ax^{2^k}+bx^{2^l}$.
\end{proof}

For $\mathrm{Tr}(x^{q+1})+L(x)$ to be a permutation polynomial of $\mathbb F_{q^n}$, it has been shown that $L$ necessarily has an inverse map on $\mathbb F_{q^n}$. It turns out that the problem involves the coefficients of the polynomial $L^{-1}$ as its inverse.

\begin{proposition}
Suppose that $L$ has an inverse map $L^{-1}$ on $\mathbb F_{q^n}$ written as
\[L^{-1}(x)=\sum_{i=0}^{m-1}L_i\big(x^{2^i}\big),\]
where $L_i$ is a $q$-linear polynomial over $\mathbb F_{q^n}$. Then $\mathrm{Tr}(x^{q+1})+L(x)$ is a permutation polynomial of $\mathbb F_{q^n}$, if and only if $L_i(1)\in\mathbb F_{q^2}$ for $0\le i<m$, if and only if $\mathbb F_q\subseteq L(\mathbb F_{q^2})$.
\end{proposition}
\begin{proof}
For $\alpha\in\mathbb F_{q^n}$, $\mathrm{Tr}_2(L^\prime(\alpha))=0$ if and only if $\alpha\in(L^\prime)^{-1}(\ker\mathrm{Tr}_2)$. Then $\ker(\mathrm{Tr}_2\circ L^\prime)=(L^\prime)^{-1}(\ker\mathrm{Tr}_2)$, and $(L^\prime)^{-1}(\ker\mathrm{Tr}_2)\subseteq\ker\mathrm{Tr}$ if and only if
\[\begin{split}0&=\mathrm{Tr}\big((L^\prime)^{-1}\big(x^{q^2}+x\big)\big)\\&=\mathrm{Tr}\big((L^{-1})^\prime\big(x^{q^2}+x\big)\big)\\&=\sum_{i=0}^{m-1}\mathrm{Tr}\big(L_i(1)\big(x^{q^2}+x\big)\big)^{2^{-i}}\\&=\sum_{i=0}^{m-1}\mathrm{Tr}\big(\big(L_i(1)^{q^{n-2}}+L_i(1)\big)x\big)^{2^{-i}}.\end{split}\]
If $L_i(1)\in\mathbb F_{q^2}$ for $0\le i<m$, then $L^{-1}(\mathbb F_q)\subseteq\mathbb F_{q^2}$ since
\[L^{-1}(x)\equiv\sum_{i=0}^{m-1}L_i(1)x^{2^i}\pmod{x^q+x}.\]
The converse is also true, as a result of the one-to-one correspondence between the polynomials over $\mathbb F_{q^2}$ of degree less than $q$ and the maps from $\mathbb F_q$ to $\mathbb F_{q^2}$.
\end{proof}

\begin{corollary}
Assume that $\mathrm{Tr}(x^{q+1})+L(x)$ is a permutation polynomial of $\mathbb F_{q^n}$ and $\lambda$ is a $2$-linear permutation polynomial of $\mathbb F_{q^n}$. If $\lambda(\mathbb F_{q^2})=\mathbb F_{q^2}$, then both $\mathrm{Tr}(x^{q+1})+\lambda(x)$ and $\mathrm{Tr}(x^{q+1})+L(\lambda(x))$ are permutation polynomials of $\mathbb F_{q^n}$. If $\lambda(\mathbb F_q)=\mathbb F_q$, then $\mathrm{Tr}(x^{q+1})+\lambda(L(x))$ is a permutation polynomial of $\mathbb F_q$.
\end{corollary}

\section{Conclusions}

We have studied permutation polynomials over finite fields of even characteristic, specifically those of the form $\mathrm{Tr}(Ax^{q+1})+L(x)$. Employing character sums, we have established necessary and sufficient conditions for these polynomials to be permutation polynomials of $\mathbb F_{q^n}$. In addition, the special cases where $L$ is a monomial or binomial are discussed in detail. For further research, one may explore more different forms of $2$-linear polynomials satisfying the specific conditions, or the properties of those permutation polynomials as elements in the symmetric group on $\mathbb F_{q^n}$.


\begin{thebibliography}{10}

\bibitem{bhattacharya2017some}
S.~Bhattacharya and S.~Sarkar.
\newblock On some permutation binomials and trinomials over $\mathbb F_{2^n}$.
\newblock {\em Designs, Codes and Cryptography}, 82:149--160, 2017.

\bibitem{blokhuis2001permutations}
A.~Blokhuis, R.~S. Coulter, M.~Henderson, and C.~M. O'Keefe.
\newblock Permutations amongst the Dembowski-Ostrom polynomials.
\newblock In {\em Finite Fields and Applications: Proceedings of The Fifth
  International Conference on Finite Fields and Applications}, pages 37--42.
  Springer, 2001.

\bibitem{coulter1999evaluation}
R.~S. Coulter.
\newblock On the evaluation of a class of Weil sums in characteristic 2.
\newblock {\em New Zealand Journal of Mathematics}, 28(2):171--184, 1999.

\bibitem{ding2015permutation}
C.~Ding, L.~Qu, Q.~Wang, J.~Yuan, and P.~Yuan.
\newblock Permutation trinomials over finite fields with even characteristic.
\newblock {\em SIAM Journal on Discrete Mathematics}, 29(1):79--92, 2015.

\bibitem{gupta2016some}
R.~Gupta and R.~Sharma.
\newblock Some new classes of permutation trinomials over finite fields with
  even characteristic.
\newblock {\em Finite Fields and Their Applications}, 41:89--96, 2016.

\bibitem{li2017new}
K.~Li, L.~Qu, and X.~Chen.
\newblock New classes of permutation binomials and permutation trinomials over
  finite fields.
\newblock {\em Finite Fields and Their Applications}, 43:69--85, 2017.

\bibitem{tu2014several}
Z.~Tu, X.~Zeng, and L.~Hu.
\newblock Several classes of complete permutation polynomials.
\newblock {\em Finite Fields and Their Applications}, 25:182--193, 2014.

\bibitem{wang2018six}
Y.~Wang, W.~Zhang, and Z.~Zha.
\newblock Six new classes of permutation trinomials over $\mathbb F_{2^n}$.
\newblock {\em SIAM Journal on Discrete Mathematics}, 32(3):1946--1961, 2018.

\bibitem{wu2013linearized}
B.~Wu and Z.~Liu.
\newblock Linearized polynomials over finite fields revisited.
\newblock {\em Finite Fields and Their Applications}, 22:79--100, 2013.

\bibitem{zha2017further}
Z.~Zha, L.~Hu, and S.~Fan.
\newblock Further results on permutation trinomials over finite fields with
  even characteristic.
\newblock {\em Finite Fields and Their Applications}, 45:43--52, 2017.

\end{thebibliography}
\end{document}